\documentclass[preprint,5p,times,twocolumn,authoryear]{elsarticle}

\usepackage{lineno,hyperref}

\usepackage[cmex10]{amsmath}
\usepackage{amssymb}
\usepackage{mathrsfs}
\usepackage{amsthm}

\usepackage{graphicx}
\usepackage{color}
\usepackage{algorithm}
\usepackage{algorithmic}
\usepackage{bm}
\usepackage{setspace}

\DeclareGraphicsExtensions{.pdf,.png,.jpg,.eps,.ps}

\def\ba{\begin{array}}
	\def\ea{\end{array}}
\def\baa{\begin{align}}
	\def\eaa{\end{align}}

\newcommand{\bsq}{\begin{subequations}}
	\newcommand{\esq}{\end{subequations}}

\newcommand{\beq}{\begin{equation}}
\newcommand{\eeq}{\end{equation}}
\newcommand{\bq}{\begin{eqnarray}}
\newcommand{\eq}{\end{eqnarray}}
\newcommand{\bqn}{\begin{eqnarray*}}
	\newcommand{\eqn}{\end{eqnarray*}}
\newcommand{\bee}{\begin{enumerate}}
	\newcommand{\eee}{\end{enumerate}}
\newcommand{\bi}{\begin{itemize}}
	\newcommand{\ei}{\end{itemize}}

\usepackage{comment}
\usepackage{ifthen}
\newboolean{showcomments}
\setboolean{showcomments}{true}
\newcommand{\wang}[1]{\ifthenelse{\boolean{showcomments}}
	{ \textcolor[rgb]{1,0,1}{(ZW:  #1)}}{}}
\newcommand{\fliu}[1]{\ifthenelse{\boolean{showcomments}}
	{ \textcolor{red}{(FL:  #1)}}{}}
\newcommand{\slow}[1]{\ifthenelse{\boolean{showcomments}}
	{ \textcolor{blue}{(SL:  #1)}}{}}

\newtheorem{theorem}{Theorem}
\newtheorem{lemma}[theorem]{Lemma}

\newtheorem{remark}{Remark}

{\newtheorem{assumption}{\textit{Assumption}}}

\allowdisplaybreaks

\journal{xxx}

\begin{document}
	
\graphicspath{{Paper_Fig/}}

\begin{frontmatter}
	\title{Exponential Stability of Partial Primal-Dual Gradient Dynamics \\with Nonsmooth Objective Functions }
	
	\tnotetext[mytitlenote]{This work was supported  by the National Natural Science Foundation of China ( No. 51677100, U1966601, U1766206).}
		
	\author[thu]{Zhaojian~Wang}
	\author[thu]{Wei Wei}
	\author[CUHK]{Changhong Zhao}
	\author[thu]{Zetian Zheng}
	\author[thu]{Yunfan Zhang}
	\author[thu]{Feng~Liu\corref{mycorrespondingauthor}}\ead{lfeng@mail.tsinghua.edu.cn}
	\cortext[mycorrespondingauthor]{Corresponding author}
	
	\address[thu]{State Key Laboratory of Power Systems, Department of Electrical Engineering, Tsinghua University, Beijing 100084, China}
	\address[CUHK]{Department of Information Engineering, The Chinese University of Hong Kong, Hong Kong}

\begin{abstract}
	In this paper, we investigate the continuous time partial primal-dual gradient dynamics (P-PDGD) for solving convex optimization problems with the form $ \min\limits_{x\in X,y\in\Omega}\ f({x})+h(y),\ \textit{s.t.}\ A{x}+By=C $, where $ f({x}) $ is strongly convex and smooth, but $ h(y) $ is strongly convex and non-smooth. Affine equality and set constraints are included. We prove the exponential stability of P-PDGD, and bounds on decaying rates are provided. Moreover, it is also shown that the decaying rates can be regulated by setting the stepsize.
\end{abstract}

\end{frontmatter}

\section{Introduction}
The primal-dual gradient dynamics (PDGD) (or saddle--point dynamics) were first introduced in \cite{arrow1958studies,kose1956solutions} and have been widely used in obtaining the primal-dual solutions of a convex (or concave) optimization problem. The partial primal-dual gradient dynamics (P-PDGD) is similar to PDGD, which is first proposed in \cite{Li:Connecting} to solve specific types of optimization problems with separable decision variables, i.e., $ \min f({x})+h(y), \ \text{s.t.}\  A{x}+By=C$. 
Typical applications of PDGD and P-PDGD include power systems \cite{Changhong:Design, Li:Connecting, mallada2017optimal, wang2019unified}, wireless communication \cite{chiang2007layering}, distributed optimization \cite{Yi:Distributed} and seeking the Nash Equilibrium in game \cite{gharesifard2013distributed}. 

Despite its wide applications, general theoretical studies of PDGD and P-PDGD are focused on its asymptotic stability analysis \cite{Feijer:Stability,Cherukuri:Asymptotic,cherukuri2018role,holding2014convergence} and exponential stability analysis \cite{cortes2019distributed,niederlander2016exponentially,dhingra2018proximal,qu2018exponential, tang2019semi,chen2019exponential, liang2019exponential,bansode2019exponential}.
In the existing literature, the global asymptotic stability of the PDGD under different settings is investigated. In \cite{Feijer:Stability}, the projection is proposed to address the inequality constraints and the PDGD is modeled as a hybrid automaton. This will result in the discontinuity of the right-hand side of primal-dual dynamics and bring difficulties in the convergence proof. Then, \cite{Cherukuri:Asymptotic} improves the convergence proof by using the invariance principle for discontinuous Caratheodory systems, which are further extended in \cite{cherukuri2018role} by using a weaker assumption to show global asymptotic stability for locally strong convex-concave Lagrangian. Extensions are given to subgradient methods in \cite{holding2014convergence}, which constrain the dynamics to a convex domain. 

Exponential stability is a desirable property of a dynamic system both theoretically and in practice. In many practical systems like the power grid, it is better to have stronger stability guarantees. In addition, a discrete-time iterative algorithm can be obtained from an exponentially stable dynamics through explicit Euler discretization, which achieves linear convergence for sufficiently small step sizes \cite{stuart1994numerical,stetter1973analysis}. 
For PDGD to solve constrained convex optimization, its locally exponential stability can be obtained by investigating spectral bounds of saddle matrices \cite{benzi2005numerical}. Regarding the global exponential stability, in \cite{qu2018exponential}, the augmented PDGD is introduced to solve convex optimization with affine inequality constraints, where the exponential stability is proved. The method is further extended in \cite{tang2019semi} for convex optimization problems with convex inequality constraints, and in \cite{chen2019exponential} for convex optimization problems with partially nonstrongly convex objective functions. In \cite{liang2019exponential}, PDGD is used to solve the distributed optimization problem with nonstrongly convex objectives, where the metrically subregular condition is adopted to prove the exponential stability. In \cite{bansode2019exponential}, a projected dynamics is proposed to solve a convex optimization problem with set and linear inequality constraints, where the exponential stability is proved on a Riemannian manifold. 
To solve nonsmooth convex optimization problems, Clark generalized gradients are adopted to replace the regular gradients. In \cite{niederlander2016exponentially,cortes2019distributed}, saddle-point-like dynamics is proposed and proved to be globally exponentially stable with equality constraints. It is further improved in \cite{dhingra2018proximal} to consider affine inequality constraints. 

This work studies the exponential stability of P-PDGD to solve optimization problems with strongly convex and non-smooth objective functions. Affine equality and set constraints are included. Compared with existing literature \cite{cortes2019distributed,niederlander2016exponentially,dhingra2018proximal,qu2018exponential, tang2019semi,chen2019exponential, liang2019exponential}, the main difference is that set constraints are considered. This is very important in practice. In many cases, set constraints are hard limits and even need to be satisfied during transient process. To this end, the projection onto the tangent cone of the set is adopted. This is also different from \cite{bansode2019exponential}, where the projection onto the set itself is used. However, the method in \cite{bansode2019exponential} does not apply to optimization problems with nonsmooth objectives. As pointed out in \cite{zeng2018distributed}, the projection of a subdifferential set cannot guarantee the existence of trajectories. This paper is also partially motivated by \cite{li2020exponentially}, which designs an exponentially convergent algorithm for the consensus problem. In this work, the optimization problem is more general and dual variables are explicitly included. 

\textit{Notations}: In this paper, use $\mathbb{R}^n$ to denote the $n$-dimensional  Euclidean space. For a column vector $ x\in \mathbb{R}^n$ (matrix $A\in \mathbb{R}^{m\times n}$), $ x^{\mathrm{T}}$($A^{\mathrm{T}}$) denotes its transpose. For vectors $ x, y\in \mathbb{R}^n$, $ x^{\mathrm{T}} y=\left\langle  x, y \right\rangle$ denotes the inner product of $ x, y$. $\left\| x \right\|=\sqrt{ x^{\mathrm{T}} x}$ denotes the Euclidean norm of $ x$. For any real symmetric matrices $P$ and $Q$, $P \succeq Q$ and $Q \preceq P$ mean that $P-Q$ is positive semidefinite.

\section{Preliminaries}\label{Preliminaries}
\subsection{Convex analysis}
Let $\Omega \subset \mathbb{R}^{n}$ be a convex set. Let $ f(x):\Omega\rightarrow\mathbb{R} $ be a locally Lipschitz continuous function and denote its Clarke generalized gradient by $\partial f( x)$ \cite[Page 27]{clarke:optimization}. For a $\mu$-strongly convex function $ f( x):\mathbb{R}^n\rightarrow\mathbb{R} $, we have $ (g_{ x}-g_{ y})^{\mathrm{T}}( x- y)\ge \mu\|x- y\|^{2},\ \forall  x,  y \in \Omega$, for all $g_{ x}\in \partial f( x)$ and $g_{ y}\in \partial f( y)$.

For $x \in \Omega,$ the tangent cone to $\Omega$ at $x$ is defined as
$$
\mathcal{T}_{\Omega}(x) \triangleq\left\{\lim _{k \rightarrow+\infty} \frac{x_{k}-x}{\tau_{k}} | x_{k} \in \Omega, x_{k} \rightarrow x, \tau_{k}>0, \tau_{k} \rightarrow 0\right\}
$$

The normal cone to $\Omega$ at $x$ is defined by
$$\mathcal{N}_\Omega(x)=\{ v|\left\langle  v,  y- x\right\rangle\le 0, \forall  y\in \Omega\}$$

By \cite{brogliato2006equivalence}, the tangent cone is the polar of the normal cone, which implies
$$ T_\Omega(x) = \{y\in R^n | \langle s,y\rangle  \le 0, \forall s \in N_\Omega (x) \} $$

Define the projection of $x$ onto $\Omega$ as 
\begin{eqnarray}
\label{def_projection}
\mathcal{P}_{\Omega}( x)=\arg \min\nolimits_{ y\in \Omega}\left\| x- y \right\|
\end{eqnarray}
We have 
\begin{align}
	x-\mathcal{P}_{\Omega}(x) \in \mathcal{N}_{\Omega}(x)
\end{align}

By \cite{brogliato2006equivalence}, the projection onto $ \mathcal{T}_{\Omega}(x) $ is computed by
\begin{align}\label{Compute_tangent}
	P_{T_{\Omega}(x)}(y)&=\lim _{\delta \rightarrow 0} \frac{P_{\Omega}(x+\delta y)-x}{\delta}\nonumber\\
	&=y-\varpi z^{*}
\end{align}
where $\varpi=\max \left\{0, \langle y,z^{*}\rangle\right\},$ and $z^{*}=\arg\max\nolimits_{\|z\|=1, z \in \mathcal{N}_{\Omega}(x)}\langle y, z\rangle$.

\subsection{Differential inclusion}
Following \cite{aubin2012differential}, a differential inclusion is given by
\begin{align}\label{Differential_inclusion}
	\dot{x}(t) \in \mathcal{F}(x(t)), x(0)=x_{0}, t \geq 0
\end{align}
where $\mathcal{F}$ is a set-valued map from points in $\mathbb{R}^{n}$ to a nonempty, compact, convex subset of $\mathbb{R}^{n}$. For each $x \in \mathbb{R}^{n},$ system \eqref{Differential_inclusion} specifies a set of possible evolutions instead of a single one. A solution of \eqref{Differential_inclusion} defined on $[0, \tau] \subset[0, \infty)$ is an absolutely continuous function $x:[0, \tau] \rightarrow \mathbb{R}^{n}$ such that \eqref{Differential_inclusion} holds for almost all $t \in[0, \tau]$ for $\tau>0$. The solution $t \mapsto x(t)$ to \eqref{Differential_inclusion} is a right maximal solution if it cannot be extended in time. Suppose that all the right maximal solutions to \eqref{Differential_inclusion} exists on $[0,+\infty)$. If $0_{m} \in \mathcal{F}\left(x^*\right)$, then $x^*$ is an equilibrium point of \eqref{Differential_inclusion}.

By \cite[Proposition 2 in p. 266, and Theorem in p. 267]{aubin2012differential}, we have 
\begin{lemma}\label{Equivalence}
	Let $\Omega$ be a closed convex subset of $\mathbb{R}^{n}$, and $\mathcal{F}$ be a map with non-empty compact value from $\Omega$ to $\mathbb{R}^{n}$. Consider two differential inclusions given by
	\begin{subequations}
		\setlength{\abovedisplayskip}{4pt}	
		\setlength{\belowdisplayskip}{4pt}
		\begin{align}\label{Normal_cone_pro}
		&\dot{x}(t) \in \mathcal{F}(x(t))-\mathcal{N}_{\Omega}(x(t)), x(0)=x_{0}\\
		\label{Tangent_cone_pro}
		&\dot{x}(t) \in P_{\mathcal{T}_{\Omega}}[\mathcal{F}(x(t))], \quad x(0)=x_{0}
		\end{align}
	The trajectory x(t) is a solution of \eqref{Normal_cone_pro} if and only if it is a solution of \eqref{Tangent_cone_pro}. Moreover, if $\mathcal{F}$ is also upper semi-continuous and bounded, there exists a solution to dynamics \eqref{Normal_cone_pro}.
	\end{subequations}
\end{lemma}

Let $V: \mathbb{R}^{n} \rightarrow \mathbb{R}$ be a locally Lipschitz continuous function and $\partial V$ be the Clarke generalized gradient of $V(x)$ at $x$. The set-valued Lie derivative $\mathcal{L}_{\mathcal{F}} V$ of $V$ with respect to \eqref{Differential_inclusion} is defined as $\mathcal{L}_{\mathcal{F}} V(x) \triangleq\left\{a \in \mathbb{R}: a=p^{\mathrm T} v, p \in \partial V(x), v \in \mathcal{F}(x)\right\}$.

\section{Problem Formulation and Algorithm Design}\label{Problem_Formulation}
\subsection{Optimization problem}
The optimization problem is 
\begin{subequations}
	\setlength{\abovedisplayskip}{4pt}	
	\setlength{\belowdisplayskip}{4pt}
	\label{OLC}     
	\begin{align}
	\min\limits_{x\in X,y\in\Omega}\quad& f({x})+h(y)
	\label{OLC_1}
	\\ 
	\text{s.t.} \quad	&A{x}+By=C
	\end{align}
\end{subequations}
where $X\subset\mathbb{R}^n$, $ \Omega\subset\mathbb{R}^m $ are compact convex sets. $A\in\mathbb{R}^{p\times n}, B\in \mathbb{R}^{p\times m}, C\in\mathbb {R}^{p}$ are constant matrices.
We make following assumptions.
\begin{assumption}\label{Assumption_Slater}
	The Slater's condition  \cite[Chapter 5.2.3]{boyd2004convex} of \eqref{OLC} holds.
\end{assumption}
\begin{assumption}\label{Assumption_f}
	For some $ \alpha >0  $, $ f(x) $ is $\alpha$-strongly convex and twice differentiable, i.e., $ \nabla^{2} f(x) \succeq \alpha I$.
\end{assumption}
\begin{assumption}\label{Assumption_g}
	The function $ h(y) $ is Lipschitz continuous, and $\beta$-strongly convex on $\Omega$ for some $ \beta>0 $, that is,
	$
	\langle y_1-y_2, g_{h}(y_1)-g_{h}(y_2)\rangle \geq \beta\|y_1-y_2\|^{2}, \ \forall y_1, y_2 \in \Omega
	$,
	where $g_{h}(y_1) \in \partial h(y_1)$ and $g_{h}(y_2) \in \partial h(y_2)$.
\end{assumption}
It should be noted that $h(y)$ could be nonsmooth.
\begin{assumption}\label{Assumption_A}
	The  matrix $A$ has full row rank and $\kappa_{1}I \preceq A A^\mathrm{T}$ for some $\kappa_{1}>0$.
\end{assumption}

\begin{remark}
	Many practical problems have the same type as problem \eqref{OLC}. For example,  distributed frequency control problem in power systems  \cite{Li:Connecting,Distributed_I:Wang,Distributed_II:Wang}, and distributed voltage control in distribution networks \cite{liu2018hybrid, wang2019asynchronous} can be generalized as \eqref{OLC}. Thus, the result in this paper can be applied to many problems in practice. 
\end{remark}

\subsection{Algorithm Design}
The Lagrangian of \eqref{OLC} is 
\begin{align}\label{Lagrangian}
	\mathop{L(x, y, \lambda)}\limits_{x\in X, y\in\Omega} & = f({x})+h(y)+\lambda^\mathrm{T}(A x+B y-C) 
\end{align}
where $ \lambda\in\mathbb{R}^p $ is the Lagrangian multiplier vector.
The partial primal-dual gradient algorithm is 
\begin{subequations}
	\setlength{\abovedisplayskip}{4pt}	
	\setlength{\belowdisplayskip}{4pt}
	\label{Controller}
	\begin{align}
	\label{Controller1}
	x&=\arg \min _{x\in X}\left\{f(x)+\lambda^\mathrm{T} A x\right\}  \\
	\label{Controller2}
	\dot y &\in P_{\mathcal{T}_{\Omega}\left(y \right)}\left(-\partial h\left(y \right)-B^\mathrm{T}\lambda\right)\\
	\dot\lambda&=Ax+By-C
	\end{align}
\end{subequations}

Define functions
\begin{align}
	\varphi(\lambda) & \triangleq \min _{x\in X}\left\{f(x)+\lambda^\mathrm{T} A x\right\} \\ 
	\label{Partial_Lagrangian}
	\mathop	{\hat{L}(y, \lambda)}\limits_{y\in\Omega} & \triangleq \varphi(\lambda)+h(y)+\lambda^\mathrm{T}(B y-C) 
\end{align}

Let $ H:=\nabla^{2} f(x) $, and its maximal eigenvalue be $ \alpha_{m} $.	Its inverse $ H^{-1} $ is also positive definite, and the minimal eigenvalue is $ \frac{1}{\alpha_{m}} $.
Then, we have the following result demonstrating properties of $ \varphi(\lambda) $ and $ \hat{L}(y, \lambda) $.
\begin{lemma}\label{Strongly_convex}
	Suppose Assumptions \ref{Assumption_f}, \ref{Assumption_A} hold. The function $\varphi(\lambda)$ is continuously differentiable and	$ \frac{\kappa_1}{\alpha_{m}} $-strongly concave. As a consequence, $ \hat{L}(y, \lambda) $ is strongly concave on $\lambda$.
\end{lemma}
\begin{proof}
	Because $f(x)$ is differentiable, $\varphi(\lambda)$ is continuously differentiable by \cite[Proposition 6.1.1]{Bertsekas2008Nonlinear}. The gradient of $\varphi(\lambda)$ is $ \nabla \varphi(\lambda)=Ax $.
	From \cite[Equation (6.9)]{Bertsekas2008Nonlinear}, the Hessian of $\varphi(\lambda)$ is
	\begin{align}\label{Hessian}
		\nabla^{2} \varphi(\lambda)=-A H^{-1} A^\mathrm{T}
	\end{align}
	For any $z\in\mathbb{R}^p$, we have 
	\begin{align}
		-z^\mathrm{T}\nabla^{2} \varphi(\lambda)z&=z^\mathrm{T}A H^{-1} A^\mathrm{T}z\nonumber\\
		&\ge \frac{1}{\alpha_{m}}z^\mathrm{T}A A^\mathrm{T}z \nonumber\\
		&\ge \frac{\kappa_1}{\alpha_{m}}z^\mathrm{T} z
	\end{align}
	where the second inequality is due to Assumption \ref{Assumption_A}. This implies that $ -\varphi(\lambda) $ is $ \frac{\kappa_1}{\alpha_{m}} $-strongly convex, i.e., $ \varphi(\lambda) $ is $ \frac{\kappa_1}{\alpha_{m}} $-strongly concave. This completes the proof.
\end{proof}
From Lemma \ref{Strongly_convex}, taking any $ \lambda_1, \lambda_2 \in \mathbb{R}^p $, we have 
\begin{align}\label{Strongly_convex_q}
	\left\langle \lambda_1-\lambda_2, \nabla \varphi(\lambda_1)-\nabla \varphi(\lambda_2)\right\rangle \leq -\frac{\kappa_1}{\alpha_{m}}\|\lambda_1-\lambda_2\|^{2}
\end{align}
Then, the algorithm \eqref{Controller} is rewritten as

\noindent \textbf{P-PDGD:}
\begin{subequations}
	\setlength{\abovedisplayskip}{4pt}	
	\setlength{\belowdisplayskip}{4pt}
	\label{Controller_PPD}
	\begin{align}
	\label{Controller_PPD1}
	x&=\arg \min _{x\in X}\left\{f(x)+\lambda^\mathrm{T} A x\right\}  \\
	\label{Controller_PPD2}
	\dot y &\in P_{\mathcal{T}_{\Omega}\left(y \right)}\left(-\partial h\left(y \right)-B^\mathrm{T}\lambda\right)\\
	\dot\lambda&=\nabla \varphi(\lambda)+By-C
	\end{align}
\end{subequations}
In the rest of the paper, we will study the properties of the algorithm \eqref{Controller_PPD}.
\begin{remark}
	The algorithm \eqref{Controller_PPD} has some significant improvement from the initial version in \cite{Li:Connecting}. First, it considers the nonsmooth objective functions, where the Clark generalized gradient is utilized. Second, the set constraint is included, where the projection onto the tangent cone is adopted. Thus, the proposed algorithm is more general.
	
	In \eqref{Controller_PPD2}, the projection onto $ \mathcal{T}_{\Omega}(y) $ is applied. The reason is that the projection from	$\partial h\left(y\right)$ onto $\Omega$ may be a nonconvex differential inclusion. As a result, the existence of trajectories of \eqref{Controller_PPD} is not guaranteed, which sets a difficult barrier for convergence analysis.
\end{remark}

\section{Optimality and Exponential Convergence}
In this section, we first investigate the optimality of the equilibrium point of \eqref{Controller_PPD}. Then, we prove that the trajectory of dynamics \eqref{Controller_PPD} converges to the equilibrium point exponentially.
\subsection{Optimality}
Before proving the optimality, we introduce the existence of solutions to \eqref{Controller_PPD}.
\begin{lemma}\label{bounded}
	Suppose Assumptions \ref{Assumption_f}, \ref{Assumption_g}, and \ref{Assumption_A} hold. Consider dynamics \eqref{Controller_PPD}. 
	\begin{enumerate}
		\item If ${y}(0) \in \Omega$, then $y(t) \in \Omega$ for all $t \geq 0$.
		\item If $\lambda(0)$ is bounded, then $\lambda(t) $ is bounded for all $t \geq 0$.
		\item There exists a solution to \eqref{Controller_PPD}.
	\end{enumerate}
\end{lemma}
\begin{proof}
	For 1), define a function  
	\begin{align}
	E\left(y(t)\right)=\frac{1}{2}\left\|y(t)-P_{\Omega}\left(y(t)\right)\right\|^{2}
	\end{align}
	Its gradient is \cite[Theorem 1.5.5]{facchinei2003finite}
	\begin{align}
	\nabla E\left(y(t)\right)=y(t)-P_{\Omega}\left(y(t)\right)
	\end{align}
	The time derivative of $ E\left(y(t)\right) $ is
	\begin{align}
	\dot E\left(y(t)\right)=\left\langle\nabla E\left(y(t)\right),\dot{y}(t)\right\rangle=\left\langle y(t)-P_{\Omega}\left(y(t)\right),\dot{y}(t)\right\rangle
	\end{align}
	Because $ y(t)-P_{\Omega}(y(t)) \in \mathcal{N}_{\Omega}(y(t)) $ and $ \dot{y}(t) \in \mathcal{T}_{\Omega}\left(y(t)\right) $, we have $ \dot E\left(y(t)\right)\le0 $. This implies that $ E\left(y(t)\right) $ is non-increasing. From $ E\left(y(t)\right)\ge0,\ \forall t\ge0 $ and $ E\left(y(0)\right)=0 $, we have $ E\left(y(t)\right)=0,\ \forall t\ge0  $, i.e., $ y(t)=P_{\Omega}\left(y(t)\right), \ \forall t\ge0 $. 
	
	For 2), define a function
	\begin{align}
		\tilde E\left(\lambda(t)\right)=\frac{1}{2}\left\|\lambda(t)\right\|^{2}
	\end{align}
	The time derivative of $\tilde E\left(\lambda(t)\right)$ along \eqref{Controller_PPD} is 
	\begin{align}
		\dot{\tilde E}(\lambda) &=\lambda^{\rm T}(\nabla \varphi(\lambda)+By-C)\nonumber\\
		&=\lambda^{\rm T}\nabla \varphi(\lambda)+\lambda^{\rm T}(By-C)\nonumber\\
		&\le -\frac{\kappa_1}{\alpha_{m}}\|\lambda\|^{2} + \lambda^{\rm T}\nabla \varphi(0)+\lambda^{\rm T}(By-C)\nonumber\\
		&\le -\frac{\kappa_1}{\alpha_{m}}\|\lambda\|^{2} + a_\lambda\|\lambda\| \nonumber\\
		&= -\frac{2\kappa_1}{\alpha_{m}} \tilde E\left(\lambda\right)+a_\lambda\sqrt{2\tilde E\left(\lambda\right)}
	\end{align}
	where $ a_\lambda=\max_{y\in\Omega}\left(\|\nabla \varphi(0)\|+\|By-C\|\right) $, and the first inequality is due to $ \left\langle \lambda, \nabla \varphi(\lambda)\right\rangle \leq  -\frac{\kappa_1}{\alpha_{m}}\|\lambda\|^{2} + \lambda^{\rm T}\nabla \varphi(0)$. The second inequality is due to the boundedness of $y(t)$.
	Then, we have $ \|\tilde{E}(\lambda(t))\|\le \max\left\{\tilde{E}(\lambda(0)),\frac{{{\alpha ^2_m}{a^2_\lambda }}}{{2{\kappa^2 _1}}}\right\} $. Thus, $\tilde{E}(\lambda(t)), t \geq 0$ is bounded, so is $\lambda(t), t \geq 0$. 
	
	Because $h(y)$ is Lipschitz, $ \partial h\left(y \right) $ is nonempty, compact, convex, and upper semicontinuous \cite[Proposition 6]{cortes2008discontinuous}. 
	From 1) and 2), we have $(y(t),\lambda(t))$ is bounded. Then, by Lemma \ref{Equivalence}, we can prove 3). 
	
	This completes the proof.	
\end{proof}

\begin{remark}
	Lemma \ref{bounded} shows that the trajectory of $y(t)$ will stay in $\Omega$ as long as $y(0)\in \Omega$. This is very important in practice besides paving the way for convergence proof. Many domain constraints are hard limits, which cannot be violated even in the transient process \cite{trip2019optimal}. For example, the power generation limits of generators and capacity limits of inverters cannot be violated physically \cite{Distributed_I:Wang,Distributed_II:Wang}. The voltage limits should not be violated, otherwise, it is dangerous for system operators \cite{wang2019unified}. Thus, the results in the paper can be applied to many practical problems.
\end{remark}

Let $ (x^*, y^*, \lambda^* ) $ be an equilibrium of \eqref{Controller_PPD}. Then 
\begin{subequations}\label{Equilibrium}
	\begin{align}
	\label{Equilibrium1}
	x^*&=\arg \min _{x\in X}\left\{f(x)+(\lambda^*)^\mathrm{T} A x\right\}  \\
	\label{Equilibrium2}
	0 &\in P_{\mathcal{T}_{\Omega}\left(y^* \right)}\left(-\partial h\left(y^* \right)-B^\mathrm{T}\lambda^*\right)\\
	\label{Equilibrium3}
	0&=Ax^*+By^*-C
	\end{align}
\end{subequations}

\begin{theorem}\label{Optimality}
	Suppose Assumptions \ref{Assumption_Slater}, \ref{Assumption_f}, \ref{Assumption_g} hold. The point $ (x^*, y^*, \lambda^* ) $ satisfies \eqref{Equilibrium}, if and only if it is the primal-dual optimal solution to \eqref{OLC} and its dual problem.
\end{theorem}
\begin{proof}
	According to the Karush-Kuhn-Tucker (KKT) optimal conditions \cite[Theorem 3.25]{ruszczynski2006nonlinear}, the primal-dual optimal solution should satisfy
	\begin{subequations}\label{KKT}
		\begin{align}
		\label{KKT1}
		0&\in\nabla f(x^*)+A^\mathrm{T}\lambda^*+\mathcal{N}_{X}(x^*) \\
		\label{KKT2}
		0 &\in \partial h\left(y^* \right)+B^\mathrm{T}\lambda^*+\mathcal{N}_{\Omega}(y^*)\\
		\label{KKT3}
		0&=Ax^*+By^*-C
		\end{align}
	\end{subequations} 
	Compare \eqref{KKT} with \eqref{Equilibrium}, and we know \eqref{KKT1} is equivalent to \eqref{Equilibrium1}. By Lemma \ref{Equivalence}, we have \eqref{KKT2} is equivalent to \eqref{Equilibrium2}. Thus, \eqref{KKT} is equivalent to \eqref{Equilibrium}. Because the optimization problem \eqref{OLC} is convex and with strongly convex objective functions, $ (x^*, y^*, \lambda^* ) $ is the primal-dual optimal solution to \eqref{OLC} and its dual problem. This completes the proof. 
\end{proof}

\subsection{Exponential convergence}

In this subsection, we analyze the convergence rate for algorithm dynamics \eqref{Controller_PPD}.

\begin{theorem}\label{Convergence}
	Suppose Assumptions \ref{Assumption_Slater}, \ref{Assumption_f}, \ref{Assumption_g} hold.  The solution algorithm dynamics \eqref{Controller_PPD} converges to its equilibrium point $ (x^*, y^*, \lambda^* ) $ exponentially.
\end{theorem}

\begin{proof}
Define the Lyapunov function candidate
\begin{align}
	V(y,\lambda)=\frac{1}{2}\|y-y^{*}\|^2+ \frac{1}{2}\|\lambda-\lambda^{*}\|^2
\end{align}

The time derivative of $ V(y,\lambda) $ is
\begin{align}
	&\mathcal{L}_{\mathcal{F}} V(y, \lambda)=\left\{a \in \mathbb{R}: \right.\nonumber\\ 
	&\qquad  a=\nabla_{y} V(y, \lambda)^{\mathrm{T}} P_{\mathcal{T}_{\Omega}\left(y(t)\right)}\left(-\partial h\left(y(t)\right)-B^\mathrm{T}\lambda\right) \nonumber\\ 
	&\qquad \left. +\nabla_{\lambda} V(y, \lambda)^{\mathrm{T}}(\nabla \varphi(\lambda)+By-C)\right\}
\end{align}

Suppose $a \in \mathcal{L}_{\mathcal{F}} V(y, \lambda)$. There is $\eta(y) \in \partial h(y)$ such that
\begin{align} 	
	a=&\left(y-y^{*}\right)^\mathrm{T}P_{\mathcal{T}_{\Omega}\left(y(t)\right)}\left(-\eta(y)-B^\mathrm{T}\lambda\right)\nonumber \\ 
	&+\left(\lambda-\lambda^{*}\right)^\mathrm{T}(\nabla \varphi(\lambda)+By-C)
\end{align}

From \eqref{Normal_cone_pro}, we have 
\begin{align}
	-\eta(y)-B^\mathrm{T}\lambda-\dot{y} \in \mathcal{N}_{\Omega}(y)
\end{align}
From the definition of $ \mathcal{N}_{\Omega}(y(t)) $ and the fact that $ y^* \in \Omega $, we have 
\begin{align}
	&\left\langle y^{*}-y,-\eta(y)-B^\mathrm{T}\lambda-\dot{y}\right\rangle \le 0\nonumber\\
	&\left\langle y-y^{*},\dot{y}\right\rangle \le \left\langle y-y^{*},-\eta(y)-B^\mathrm{T}\lambda\right\rangle
\end{align}

Then
\begin{align}\label{Inequality3}
	a\le&\left(y-y^{*}\right)^\mathrm{T}\left(-\eta(y)-B^\mathrm{T}\lambda\right)\nonumber \\ 
	&+\left(\lambda-\lambda^{*}\right)^{{T}}(\nabla \varphi(\lambda)+By-C)
\end{align}
From \eqref{KKT}, the definition of $ \mathcal{N}_\Omega(x) $, and the fact that $y(t)\in\Omega, \forall t$, at the equilibrium, we have 
\begin{subequations}
	\setlength{\abovedisplayskip}{4pt}	
	\setlength{\belowdisplayskip}{4pt}
	\label{Eqlibr}
	\begin{align}
	\label{Eqlibr1}
	&\left\langle-\eta\left(y^*\right)-B^\mathrm{T}\lambda^*,y-y^* \right\rangle\le 0\\
	\label{Eqlibr2}
	&0=\nabla \varphi(\lambda^*)+By^*-C
	\end{align}
\end{subequations}

Combine \eqref{Eqlibr} with \eqref{Inequality3}, and we have 
\begin{align} 	
	a&\le\left(y-y^{*}\right)^\mathrm{T}\left(-(\eta(y)-\eta\left(y^*\right))-B^\mathrm{T}(\lambda-\lambda^*)\right)\nonumber \\ 
	&\quad+\left(\lambda-\lambda^{*}\right)^{{T}}(\nabla \varphi(\lambda)-\nabla \varphi(\lambda^*)+By-By^*)\nonumber \\ 
	&\le -\left(y-y^{*}\right)^\mathrm{T}\left( \eta(y)-\eta\left(y^*\right )\right)\nonumber\\
	&\quad+\left(\lambda-\lambda^{*}\right)^{{T}}(\nabla \varphi(\lambda)-\nabla \varphi(\lambda^*)) \nonumber\\
	&\le -\beta\|y-y^{*}\|^{2}-\frac{\kappa_1}{\alpha_{m}}\|\lambda-\lambda^{*}\|^{2}\nonumber\\
	&\le -\gamma V
\end{align}
where $\gamma=\min\left\{2\beta,\ 2\frac{\kappa_1}{\alpha_{m}}\right\}$. As a result, $V(t) \leq V(0) e^{-\gamma t}$, and we have $ \left\| \left({y\left( t \right) - {y^*}}, {\lambda \left( t \right) - {\lambda ^*}}\right) \right\|\le \sqrt{2V(0)} e^{-\frac{\gamma}{2} t} $ converges to $(y^*,\lambda^*)$ exponentially, and convergence rate is no less than $ \frac{\gamma}{2} $. This completes the proof.
\end{proof}
\begin{remark}[Decaying rate]\label{Decaying_rate}
	Give some $\tau>0$, the P-PDGD can be written as
	\begin{subequations}
		\setlength{\abovedisplayskip}{4pt}	
		\setlength{\belowdisplayskip}{4pt}
		\label{Controller_PPD_tau}
		\begin{align}
		\label{Controller_PPD1_tau}
		x&=\arg \min _{x\in X}\left\{f(x)+\lambda^\mathrm{T} A x\right\}  \\
		\label{Controller_PPD2_tau}
		\dot y &\in \tau P_{\mathcal{T}_{\Omega}\left(y \right)}\left(-\partial h\left(y \right)-B^\mathrm{T}\lambda\right)\\
		\dot\lambda&=\tau(\nabla \varphi(\lambda)+By-C)
		\end{align}
	\end{subequations}
	Then, the Lyapunov function is 
	\begin{align}
		V_2(y,\lambda)=\frac{1}{2\tau}\|y-y^{*}\|^2+ \frac{1}{2\tau}\|\lambda-\lambda^{*}\|^2
	\end{align}
	
	Follow the similar analysis in Theorem \ref{Convergence}, and we have 
	\begin{align} 	
	a&\le -\frac{1}{2}\gamma \left\| \left({y\left( t \right) - {y^*}}, {\lambda \left( t \right) - {\lambda ^*}}\right) \right\|^2 \nonumber\\
	&\le -\gamma\tau V_2(y,\lambda)
	\end{align}
	As a result, $V_2(t) \leq V_2(0) e^{-\gamma\tau t}$, and we have $ \left\| \left({y\left( t \right) - {y^*}}, {\lambda \left( t \right) - {\lambda ^*}}\right) \right\|\le \sqrt{2\tau V_2(0)} e^{-\frac{\gamma\tau}{2}t} $. Thus, the decaying rate bound can be regulated by $\tau$. If $\tau=1$, it will be same as that in Theorem \ref{Convergence}. 
\end{remark}

\section{Illustrative Examples}
We consider a convex problem that models the optimal voltage control in distribution networks. The problem is formulated as
\begin{subequations}\label{eq_opt2}  
	\setlength{\abovedisplayskip}{4pt}	
	\setlength{\belowdisplayskip}{4pt}
	\begin{align}
	\min\limits_{{U}, {q}\in \mathbb{R}^n} &\quad   f=\frac{a}{2}\|{U}-\textbf{1}\|^2 + \sum\nolimits_{j=1}^n h_j(q_j)
	\label{eq_opt2a}
	\\ 
	\text{s.t.}  
	&  \quad	{B}{U}={q}+C
	\label{eq_opt2b} \\
	&  \quad	\underline{q}\le {q}\le \overline{q} \label{eq_opt2f} 
	\end{align}   
\end{subequations}
where $ U $ is the voltage, $q$ is the reactive power. $ {B}{U}={q}+C $ is the linear model derived from the DistFlow equations \cite{baran1989optimal}. $ \underline{q}, \overline{q} $ are lower and upper bound of $q$. $a$ is a real positive constant. The first part of objective function, $ \frac{a}{2}\|{U}-\textbf{1}\|^2 $, is the voltage difference, where $ \textbf{1} $ is the nominal voltage. The second part, $ h_j(q_j) $, is the regulation cost of reactive power, which is strongly convex and nonsmooth, defined by  
\begin{equation}
	\setlength{\abovedisplayskip}{4pt}	
	\setlength{\belowdisplayskip}{4pt}
	\begin{aligned}
	h_j(q_j)=\left\{ \begin{array}{l}
	q_j^2-0.02,\quad\ \ q_j\le -0.2\\
	\frac{1}{2}q_j^2 ,\quad \quad \qquad -0.2<q_j\le  0.2\\
	q_j^2-0.02,\quad\ \ 0.2<q_j
	\end{array} \right.
	\end{aligned}
\end{equation}
An 8-bus feeder is utilized as the test system, the detailed description is given in \cite{wang2019asynchronous}. The bus 0 is the substation with voltage as $1$ and has no generator. Other have reactive power regulation capability with  $\overline{{q}}=-\underline{{q}}=[80, 80, 88, 80, 104, 80, 96]$kVar. We set $a=8$, $ C=(1.011,-0.009,-0.1,0.14,-0.26,-0.019,-0.06)^{\rm T} $. The minimal eigenvalue of $B$ is $0.1165$. Moreover, we have $ \beta = 1, \kappa_1 = 0.1165^2 = 0.01357, \alpha_m = a = 8 $. Therefore $ \gamma = \min\{2\beta, 2\kappa_1/\alpha_m\} = \min\{2, 2\times0.01357/8\} = 0.00339 $.

Define $ \Omega\triangleq \{q:\underline{q}\le {q}\le \overline{q}\} $, and then the P-PDGD for \eqref{eq_opt2} is 
\begin{subequations}
	\setlength{\abovedisplayskip}{4pt}	
	\setlength{\belowdisplayskip}{4pt}
	\label{Controller_V}
	\begin{align}
	\label{Controller_V1}
	U&=\arg \min _{U}\left\{\frac{a}{2}\|U-\textbf{1}\|^2+\lambda^\mathrm{T}BU\right\}=-\frac{1}{a}B^\mathrm{T}\lambda+\textbf{1}  \\
	\label{Controller_V2}
	\dot q &\in P_{\mathcal{T}_{\Omega}\left(q \right)}\left(-\partial h\left(q \right)+\lambda\right)\\
	\dot\lambda&=-\frac{1}{a}B B^{\rm T}\lambda-q+B\times\textbf{1}-C
	\end{align}
\end{subequations}
The method in \eqref{Compute_tangent} is utilized to compute the projection onto $ {\mathcal{T}_{\Omega}\left(q \right)} $. The simulations are implemented on Matlab R2013b, and the function ``ode23tb" is adopted to solve \eqref{Controller_V}. 
The equilibrium point of \eqref{Controller_V} is
$$ q^*=(0.031, 0.031, 0.042, 0.075, 0.104, 0.035, 0.044)^{\rm T}, $$
$$U^*=(0.992, 1.003, 1.004, 1.031, 0.959, 1.004, 0.997)^{\rm T}.$$

The dynamic performance of \eqref{Controller_V} are given in Fig.\ref{Reactive_power} and Fig.\ref{Exponential_figure}. In Fig.\ref{Reactive_power}, the dynamics of reactive power are illustrated, which shows that $q_i(t), \forall i$ is within its lower and upper limits. This validates Lemma \ref{bounded}.
\begin{figure}[t]
	\setlength{\abovecaptionskip}{0pt}
	\centering
	\includegraphics[width=0.45\textwidth]{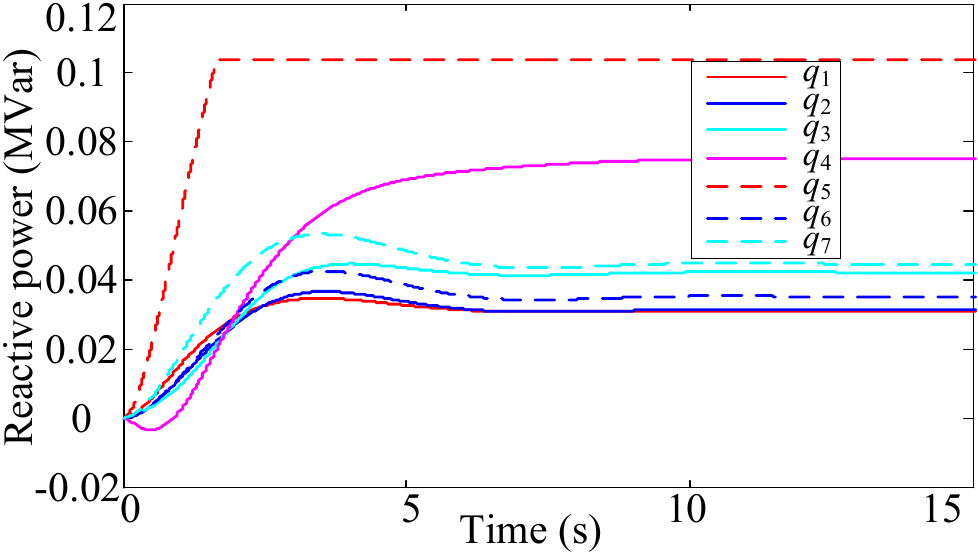}
	\caption{Dynamics of reactive power $ q $ at each bus}
	\label{Reactive_power}
\end{figure}

In Fig.\ref{Exponential_figure}, the convergence of $ \left\| \left({q\left( t \right) - {q^*}}, {\lambda \left( t \right) - {\lambda ^*}}\right) \right\| $ is illustrated. The dotted line is the trajectory of $ \sqrt {{{\left\| {q\left( 0 \right) - {q^*}} \right\|}^2} + {{\left\| {\lambda \left( 0 \right) - {\lambda ^*}} \right\|}^2}} e^ { - \gamma t/2}   $, which is the upper bound of exponential convergence obtained from Theorem \ref{Convergence}. The blue line is the trajectory of $ (q(t),\lambda(t)) $. It is shown that $ (q(t),\lambda(t)) $ converges to $ (q^*,\lambda^*) $ rapidly. The convergence speed is much faster than computed upper bound.
\begin{figure}[t]
	\setlength{\abovecaptionskip}{0pt}
	\centering
	\includegraphics[width=0.45\textwidth]{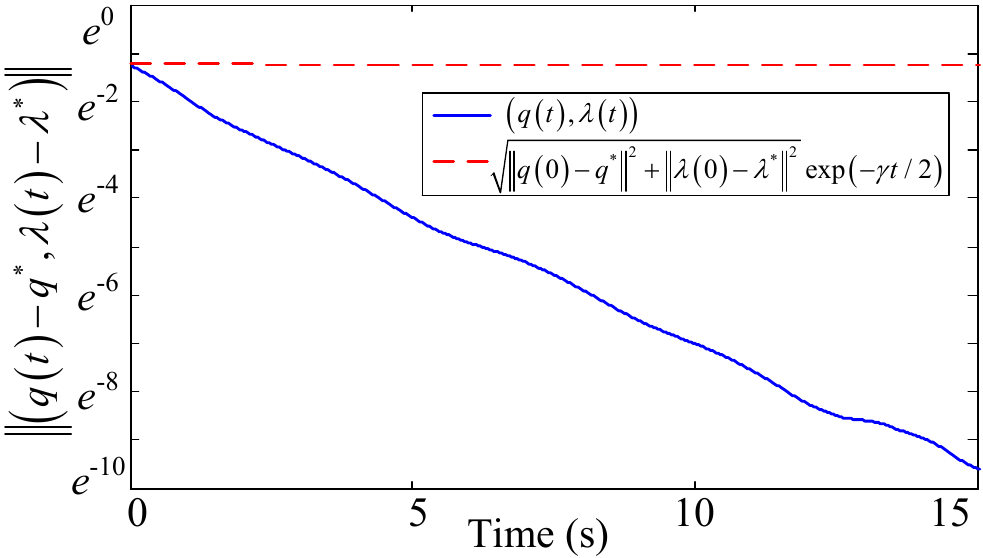}
	\caption{Illustration of the relative distances to $ (q^*, \lambda^*) $}
	\label{Exponential_figure}
\end{figure}

We further investigate the decaying rates under different $\tau$, which is illustrated in Fig.\ref{Exponential_different_tau}. With the increasing of $\tau$, the decaying rates also increase, and the exponential convergence always holds. The result is consistent with the analysis in Remark \ref{Decaying_rate}. 
\begin{figure}[t]
	\setlength{\abovecaptionskip}{0pt}
	\centering
	\includegraphics[width=0.45\textwidth]{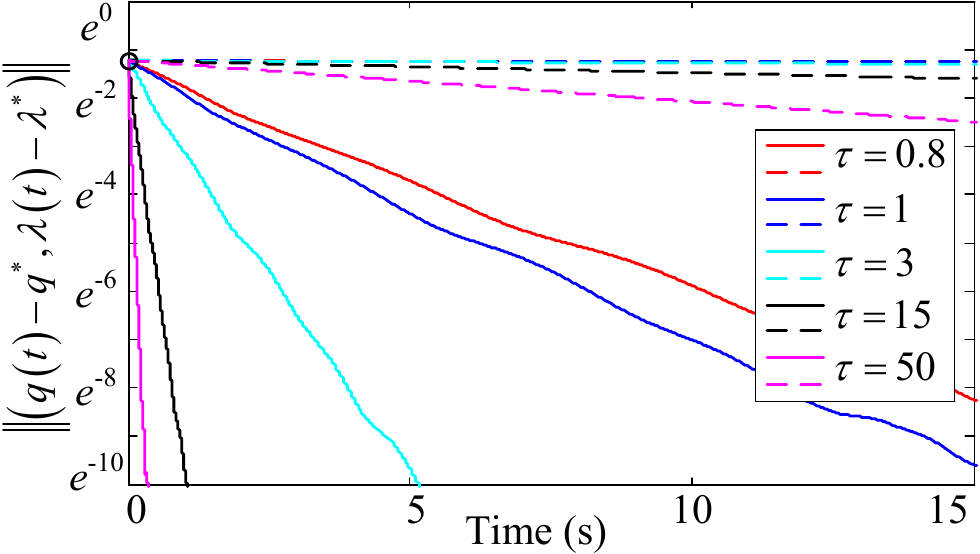}
	\caption{Illustration of the decaying rate with different $\tau$. The dotted lines are trajectories of $ \sqrt{2\tau V_2(0)} e^{-\frac{\gamma\tau}{2}t} $ with different $\tau$. The solid lines are the corresponding trajectories of $ \left\| \left({y\left( t \right) - {y^*}}, {\lambda \left( t \right) - {\lambda ^*}}\right) \right\| $. }
	\label{Exponential_different_tau}
\end{figure}

\section{Conclusion}\label{Conclusion}
This paper investigates the P-PDGD for solving convex optimization with strongly convex and non-smooth objectives. Affine equality and set constraints are considered. We prove the exponential stability of P-PDGD, where bounds on decaying rates are also provided. It is also validated that the algorithm is almost initialization free as long as the initial point satisfies the set constraints. 

Our results are promising in many practical problems, such as the frequency and voltage control in power systems, which can provide a stronger stability guarantee. However, there are still some limitations on the problem form. In the future, we will investigate exponentially convergent algorithms for more general optimization problems.


\bibliography{mybib}

\end{document}